\documentclass[11pt]{article}
\usepackage{mymacros}

\usepackage{textcomp,geometry,graphicx,enumerate,fullpage}
\usepackage{color}
\usepackage{booktabs}

\title{Maximizing Products of Linear Forms, and \\
  The Permanent of Positive Semidefinite Matrices}

\author{Chenyang Yuan \and Pablo A. Parrilo}

\newtheorem{theorem}{Theorem}
\numberwithin{theorem}{section}
\newtheorem{definition}[theorem]{Definition}
\newtheorem{remark}[theorem]{Remark}
\newtheorem{lemma}[theorem]{Lemma}
\newtheorem{corollary}[theorem]{Corollary}
\newtheorem{proposition}[theorem]{Proposition}

\newtheorem{fact}[theorem]{Fact}

\newtheorem{conjecture}[theorem]{Conjecture}

\DeclareMathOperator{\rel}{rel}

\DeclareMathOperator{\GammaText}{Gamma}
\newcommand{\cnormal}{\mathcal{C}\mathcal{N}}
\newcommand{\iid}{i.i.d.\ }

\begin{document}
\maketitle

\begin{abstract}
  We study the convex relaxation of a polynomial optimization problem,
  maximizing a product of linear forms over the complex sphere. We show that
  this convex program is also a relaxation of the permanent of Hermitian
  positive semidefinite (HPSD) matrices. By analyzing a constructive randomized
  rounding algorithm, we obtain an improved multiplicative approximation factor
  to the permanent of HPSD matrices, as well as computationally efficient
  certificates for this approximation. We also propose an analog of van der
  Waerden's conjecture for HPSD matrices, where the polynomial optimization
  problem is interpreted as a relaxation of the permanent.
\end{abstract}

\section{Introduction}
We study the problem of maximizing a product of linear forms on the complex
$(n-1)$-sphere of radius $\sqrt n$:
\begin{align} \label{eq:product-of-linear-forms}
  r(A) \equiv \max_{\norm{x}^2=n} \prod_{i=1}^n \abs{\dotp{x, v_i}}^2,
\end{align}
where $A = V^\dagger V$ and $v_i$ are the columns of $V$. We show that the
natural convex relaxation of \eqref{eq:product-of-linear-forms},
\begin{align} \label{eq:convex-rel}
  \max \, \prod_{i=1}^n v_i^\dagger P v_i \,\mbox{ s.t. }\,
  \Tr(P) = n,\, P \succeq 0,
\end{align}
is also a relaxation of the permanent of $A$, which is defined by
\begin{align} \label{eq:def-per}
  \per(A) = \sum_{\sigma \in \mathbf{S}_n} \prod_{i=1}^n A_{i,\sigma(i)},
\end{align}
where the sum is over all $n!$ permutations of $n$ elements. Computing the
permanent exactly is \#P-hard \cite{Valiantcomplexitycomputingpermanent1979},
and approximation efforts have been focused on classes of matrices with
computationally efficient certificates of permanent non-negativity. For matrices
with non-negative entries, \cite{JerrumPolynomialtimeApproximationAlgorithm2004}
gave a randomized algorithm achieving a $(1+\epsilon)$-approximation. There has
been recent interest in approximating the permanent of HPSD matrices due to
their applications in quantum information \cite{GrierNewhardnessresults2018}.
The work by \cite{AnariSimplyExponentialApproximation2017a} gave the first
polynomial-time algorithm for approximating the permanent of HPSD matrices with
a simply exponential multiplicative approximation factor of
$\frac{n!}{n^n}e^{-n\gamma}$, where $\gamma \approx 0.577$ is the
Euler-Mascheroni constant. Their algorithm is based on the following convex
program relaxation of the permanent.
\begin{definition}
  Given a HPSD matrix $A \in \C^{n \times n}$, we define $\rel(A)$ as the
  solution to the optimization problem:
  \begin{align}\label{eq:rel-opt-problem}
    \rel(A) \equiv \left\{
    \begin{array}{rl}
      \min \,& \prod_{i=1}^n D_{ii} \\
      \text{s.t.} & A \preceq D, \quad D \text{ is diagonal}
    \end{array}
                    \right.
  \end{align}
\end{definition}
In this paper, we show that $\rel(A)$ is equivalent to the convex relaxation
\eqref{eq:convex-rel}. Our main result, Theorem
\ref{thm:approx-factor-prod-linear-complex}, uses this connection between the
polynomial optimization problem \eqref{eq:product-of-linear-forms} and $\per(A)$
to provide a new analysis of the approximation of $\per(A)$ in terms of the rank
of the optimal solution to \eqref{eq:convex-rel}. By bounding this rank, we
prove an improved approximation factor for all finite $n$:
\begin{corollary} \label{cor:approx-factor} Given a HPSD matrix
  $A \in \C^{n \times n}$, $\rel(A)$ is an
  $\frac{n!}{n^n}e^{-nL_r}$-approximation to $\per(A)$:
  \begin{align*}
    \frac{n!}{n^n}e^{-nL_r} \rel(A) \le \per(A) \le \rel(A)
  \end{align*}
  where $r = O(\sqrt n)$, $L_r = H_{r-1} - \log(r)$, and
  $H_r = \sum_{k=1}^{r} \frac{1}{k}$ is the $r$-th harmonic number.
\end{corollary}
From the definition of the Euler-Mascheroni constant,
$\lim_{n \rightarrow \infty} L_r = \gamma$. For any finite $n$, $L_r < \gamma$
and thus $\frac{n!}{n^n}e^{-nL_r} > \frac{n!}{n^n}e^{-n\gamma}$. More precisely,
using Proposition \ref{prop:approx-asymptotics}, we can show that this is a
$e^{O(\sqrt n)}$ multiplicative improvement.
\cite{AnariSimplyExponentialApproximation2017a} also constructed a series of
matrices $A_k$ such that $(\rel(A_k)/\per(A_k)))^{1/n} \rightarrow e^{1+\gamma}$
as $k \rightarrow \infty$. However since this result only rules out improvements
on the order of $e^{O(n)}$, it does not contradict Corollary
\ref{cor:approx-factor}.

In Section \ref{sec:proof} we analyze the convex relaxation of
\eqref{eq:product-of-linear-forms}, describe a rounding procedure and prove its
approximation factor. In particular we prove that
\begin{align*}
  e^{-nL_r} \rel(A) \le r(A) \le \rel(A).
\end{align*}
In Section \ref{sec:permanent} we prove Theorem
\ref{thm:approx-factor-prod-linear-complex}. We first show that the convex
relaxation of $r(A)$ is equivalent to $\rel(A)$. Then using the vector produced
by the rounding procedure of the relaxation, we construct a rank-1 matrix whose
permanent lower bounds $\per(A)$, thus showing that $\rel(A)$ also
well-approximates $\per(A)$. Note that in
\cite{AnariSimplyExponentialApproximation2017a} only the existence of this
rank-1 matrix is shown, but in our analysis we provide an explicit construction
of a rank-1 matrix whose permanent lower bounds $\per(A)$. This combined with
the diagonal matrix in \eqref{eq:rel-opt-problem} whose permanent upper-bounds
$\per(A)$ certifies the approximation. In Section \ref{sec:conj} we explore
reasons why the convex relaxation of \eqref{eq:product-of-linear-forms} is
equivalent to $\rel(A)$. We conjecture that \eqref{eq:product-of-linear-forms}
is itself a $\frac{n!}{n^n}$ approximation to $\per(A)$, explain why it is an
analogue of van der Waerden's conjecture, and show that it is implied by another
long-standing permanent conjecture.

\section{Preliminaries}
For any $x \in \C$, let $x^*$ be its complex conjugate, and $\abs{x}^2 = x
x^*$. For any matrix $A \in \C^{n \times m}$, let $A^\dagger = (A^*)^T$ be its
conjugate transpose. Given $a, b \in \C^n$, let $\dotp{a, b} = a^\dagger b$ be
the inner product on the Hilbert space $\C^n$, and $\norm{a}^2 = \dotp{a,
  a}$. Let $S_\C(n) = \{ x \in \C^n \mid \norm{x}^2 = n\}$ be the complex sphere
in $n$ dimensions of radius $\sqrt n$. A matrix $A$ is Hermitian if
$A = A^\dagger$, and is Hermitian positive semidefinite (HPSD) if in addition
$x^\dagger A x \ge 0$ for all $x \in \C^n$. We can also denote this as
$A \succeq 0$. The $\succeq$ operator induces a partial order called the
L\"owner order, where $A \succeq B$ if $A - B \succeq 0$. If $A \succeq 0$, it
can be factorized as $A = L^\dagger L$, where $L \in \C^{n \times n}$ (for
example by the Cholesky decomposition).

\subsection{Circularly-Symmetric Gaussian Random Variables}
In this paper we will use a few results involving vectors of
circularly-symmetric  valued Gaussian variables.
\begin{definition}[Circularly-symmetric Gaussian random vector]
  The complex-valued Gaussian random variable $Z = Z_r + i Z_c$ is
  circularly-symmetric if $Z_r$ and $Z_c$ are \iid drawn from
  $\mathcal{N}(0, \frac{1}{2})$. The random vector
  $\mathbf{Z} = [Z_1, \ldots, Z_n]^T$ is drawn from the distribution
  $\cnormal(0, \Sigma)$ if $Z_i$ are \iid circularly-symmetric Gaussians and
  $\Ex[\mathbf{Z}\mathbf{Z}^\dagger] = \Sigma$.
\end{definition}
The name circularly-symmetric comes from the fact that $\mathbf{Z}$ is invariant
under rotations in the complex plane, meaning that $e^{i\theta}\mathbf{Z}$ has
the same distribution as $\mathbf{Z}$ for all real $\theta$. All complex
multivariate Gaussians in this paper are circularly symmetric. Similar to real
multivariate Gaussians, a linear transform on the random vector induces a
congruence transform on the covariance matrix.
\begin{proposition}[Linear transformations of complex multivariate Gaussians]
  \label{prop:cnormal-linear-transform} Given
  $\mathbf{Z} \sim \cnormal(0, \Sigma)$ and any complex matrix $A$,
  $A\mathbf{Z}$ is also circularly symmetric and has the distribution
  $\cnormal(0, A \Sigma A^\dagger)$.
\end{proposition}
The proof of this proposition and more about complex multivariate Gaussians can
be found in \cite{GallagerCircularlySymmetricGaussianrandom}. In particular,
this tells us that $\mathbf{Z} \sim \cnormal(0, I)$ is invariant under unitary
transformations.

In the analysis of our rounding procedure, we use some results about the gamma
distribution.
\begin{fact}[Expectation of log of gamma random
  variable] \label{fact:ex-log-gamma} Let $X \sim \GammaText(\alpha, \beta)$ be
  drawn from the gamma distribution, with density
  $p(x; \alpha, \beta)=\Gamma(\alpha)^{-1} \beta^\alpha x^{\alpha-1}e^{-\beta
    x}$. Then
  \begin{align*}
    \Ex[\log X] = \psi(\alpha) - \log(\beta),
  \end{align*}
  where $\psi(x) = \frac{d}{dx} \log \Gamma(x)$ is the digamma function.
\end{fact}
This follows from the fact that the gamma distribution is an exponential family,
and $\log x$ is a sufficient statistic (see section 2.2 of
\cite{KeenerTheoreticalstatisticstopics2010} for more details). Next we prove a
useful identity.
\begin{fact} \label{fact:ex-log-CN} Let
  $[z_1, \ldots, z_r]^T \sim \cnormal(0, I_r)$, $H_n = \sum_{k=1}^n \frac{1}{k}$
  be the $n$-th harmonic number and
  $\gamma = \lim_{n\rightarrow \infty} (H_n - \log n)$ be the Euler-Mascheroni
  constant. Then
  \begin{align*}
    \Ex \log\paren{\frac{1}{r} \sum_{i=1}^r \abs{z_i}^2}
    = H_{r-1} - \gamma - \log(r)
    = L_r - \gamma.
  \end{align*}
\end{fact}
\begin{proof}
  $\sum_{i=1}^r 2\abs{z_i}^2$ is distributed as a chi-squared distribution with
  $2r$ degrees of freedom, which is equivalent to
  $\GammaText\paren{r, \frac{1}{2}}$. Using Fact \ref{fact:ex-log-gamma},
  $\Ex \log\paren{\sum_{i=1}^r \abs{z_i}^2} = \psi(r)$. Since
  $\psi(1) = -\gamma$ by Gauss's digamma theorem, the recurrence relation of the
  gamma function shows that for all positive integers $r$,
  $\psi(r) = H_{r-1} - \gamma$.
\end{proof}
Integrating a homogeneous polynomial over the complex sphere is equivalent to
taking its expectation with respect to $x \sim \cnormal(0, I)$, up to a
correction factor. This factor can be found by computing moments of a
chi-squared distribution.
\begin{fact} \label{fact:gaussian-sphere-correction} Let $p(x)$ be a degree $d$
  homogeneous polynomial in $n$ variables, $\mu_n(x)$ be the measure associated
  with the random variable $x \sim\cnormal(0, I_n)$. Then
  \begin{align*}
    \int\displaylimits_{\C^n} \abs{p(x)}^2 \,d\mu_n(x)
     = \frac{(n+d-1)!}{n^n(n-1)!} \int\displaylimits_{S_\C(n)} \abs{p(x)}^2 \,dx.
  \end{align*}
\end{fact}

\subsection{Permanent of HPSD Matrices}
One remarkable property of the permanent of HPSD matrices is that it respects
the L\"owner order. See section 2.3 of
\cite{AnariSimplyExponentialApproximation2017a} for a proof.
\begin{proposition} \label{prop:per-loewner}
  If $A \succeq B \succeq 0$, then $\per(A) \ge \per(B) \ge 0$.
\end{proposition}
We can efficiently compute the permanent of rank-1 matrices. The following
proposition immediately follows from the definition of the permanent in
\eqref{eq:def-per}.
\begin{proposition} \label{prop:per-rank-1}
  For any $v \in \C^n$, $\per(vv^\dagger) = n! \prod_{i=1}^n \abs{v_i}^2$.
\end{proposition}
The permanent of HPSD matrices also has an integral representation using complex
multivariate Gaussians. See section 4 of
\cite{BarvinokIntegrationOptimizationMultivariate2007} for more details and a
proof.
\begin{proposition} \label{prop:per-int-rep} Let $\mu_n(x)$ be the measure
  associated with the random variable $x \sim \cnormal(0, I_n)$, and $S_\C(n)$
  be the complex $(n-1)$-sphere with radius $\sqrt{n}$. For any HPSD
  $A = V^\dagger V$, where $v_i$ are the columns of $V$,
  \begin{align*}
    \per(A) = \int\displaylimits_{\C^n} \prod_{i=1}^n \abs{\dotp{v_i, x}}^2 \,d\mu_n(x)
     = \frac{(2n-1)!}{n^n(n-1)!} \int\displaylimits_{S_\C(n)} \prod_{i=1}^n \abs{\dotp{v_i, x}}^2 \,dx.
  \end{align*}
\end{proposition}

\section{Convex Relaxation and Rounding} \label{sec:proof} In this section we
analyze the convex relaxation \eqref{eq:convex-rel} and a natural rounding
algorithm for maximizing a product of linear forms over the complex sphere.
\begin{remark} \label{rmk:indep-of-factor}
  Both $r(A)$ and $\rel(A)$ are independent of the factorization of
  $A = V^\dagger V$. This is because any two different factorizations of
  $A = V_1^\dagger V_1 = V_2^\dagger V_2$ are related by a unitary transform
  $V_1 = U V_2$ for some unitary matrix $U$ \footnote{We are assuming here that
    $V_1, V_2 \in \C^{n\times n}$ even if $\rank(A) < n$, padding with zero
    columns if necessary.}. This induces a change of variables
  $x \mapsto U^\dagger x$ in \eqref{eq:product-of-linear-forms} but does not
  change the value of $r(A)$.
\end{remark}

\begin{lemma} \label{lem:primal-dual-sdp}
  Any $A \succeq 0$ can be factorized as $A = V^\dagger V$, where $v_i$ are the
  columns of $V$. Consider the following pair of convex programs:
  \begin{align}
    \label{eq:primal-sdp}
      \mu^*(A) \,\equiv\, & \min \,\, \lambda^n
      \,\mbox{ s.t. }\,
      \left\{
      \arraycolsep=1.5pt \def\arraystretch{1.1}
      \begin{array}{rl}
        V \Diag(\alpha) V^\dagger &\preceq \lambda I_n \\
        \prod_{i=1}^n \alpha_i &\geq 1 \\
        \alpha_i &> 0
      \end{array}
      \right. \\[1em]
    \label{eq:dual-sdp}
    \nu^*(A) \,\equiv\, &\max  \,\,
    \prod_{i=1}^n v_i^\dagger P v_i
    \,\mbox{ s.t. }\,
    \left\{
    \arraycolsep=1.5pt \def\arraystretch{1.1}
    \begin{array}{rl}
      \Tr(P) &= n \\
      P^\dagger &= P \\
      P &\succeq 0
    \end{array}
    \right.
  \end{align}
  Then $r(A) \le \nu^*(A) = \mu^*(A)$, thus the convex programs are relaxations of $r(A)$
  (see equation \eqref{eq:product-of-linear-forms}).
\end{lemma}
\begin{proof}
  If we add a rank-1 constraint to \eqref{eq:dual-sdp}, we get
  \eqref{eq:product-of-linear-forms}, showing that $r(A) \le \nu^*(A)$. Suppose
  we have feasible solutions $\lambda$, $\alpha_i$ and $P$ to
  \eqref{eq:primal-sdp} and \eqref{eq:dual-sdp} respectively. Then
  \begin{align*}
    \prod_{i=1}^n v_i^\dagger P v_i = \prod_{i=1}^n \alpha_i v_i^\dagger P v_i
    \le \paren{\frac{1}{n} \sum_{i=1}^n \alpha_i v_i^\dagger P v_i}^n
    \le \pfrac{\Tr(P)\lambda}{n}^n = \lambda^n,
  \end{align*}
  showing weak duality, i.e. $\nu^*(A) \le \mu^*(A)$. Since \eqref{eq:dual-sdp}
  comes from taking the dual of \eqref{eq:primal-sdp} and has a strictly
  feasible solution, strong duality holds, i.e. $\nu^*(A) = \mu^*(A)$. If
  $P=x x^\dagger$ is rank-1, then $v_i^\dagger P v_i = \abs{\dotp{x, v_i}}^2$,
  thus in \eqref{eq:dual-sdp} the variable $P$ can be interpreted as the convex
  relaxation of the rank-1 constraint in \eqref{eq:product-of-linear-forms}.
\end{proof}
Although \eqref{eq:primal-sdp} and \eqref{eq:dual-sdp} have non-linear objective
functions and are not semidefinite programs in standard form, the geometric mean
constraint/objective in them can be converted to second-order conic constraints
after a change of variables \cite{LoboApplicationssecondordercone1998}. They can
also be solved efficiently with convex programming techniques such as interior
point methods (see \cite{VandenbergheDeterminantMaximizationLinear1998}). Our
main result (Theorem \ref{thm:approx-factor-prod-linear-complex}) is proven with
the following analysis of a randomized rounding procedure to the convex
relaxation of the product of linear forms. This produces a vector that gives an
$e^{-nL_r}$-approximation to \eqref{eq:product-of-linear-forms}.
\begin{theorem} \label{thm:sdp-rounding} Given a matrix $A \succeq 0$, let
  $\nu^*(A)$ be the optimum of \eqref{eq:dual-sdp}, with optimum achieved by
  $P^* = UU^\dagger$. Suppose $P^*$ has rank $r$, therefore
  $U \in \C^{n \times r}$. If we produce a vector $y \in S_\C(n)$ using the
  following procedure:
  \begin{enumerate}
  \item Sample $z\in \C^r$ uniformly at random from the complex multivariate
    Gaussian $\cnormal(0, I_r)$
  \item Return the normalized vector $y = \sqrt{n} Uz/\norm{Uz}$
  \end{enumerate}
  Recalling that $L_r = H_{r-1} - \log r$, we have the following lower bound on
  the expected value of the objective:
  \begin{align*}
    \Ex_z\bracket{\prod_{i=1}^n \abs{\dotp{v_i, y}}^2} \ge e^{-nL_r} \nu^*(A)
  \end{align*}
\end{theorem}
\begin{proof}
  We use Jensen's inequality to bound the expectation:
  \begin{align*}
    \Ex_z\bracket{\prod_{i=1}^n\frac{n \abs{\dotp{v_i, Uz}}^2}{\norm{Uz}^2}}
    &= \Ex_z\bracket{\exp\paren{\sum_{i=1}^n (\log \abs{\dotp{v_i, Uz}}^2
                   - \log z^\dagger U^\dagger Uz + \log n)}} \\
    &\ge \exp\paren{\sum_{i=1}^n (\Ex_z \log \abs{\dotp{v_i, Uz}}^2
                      - \Ex_z \log z^\dagger U^\dagger Uz + \log n)}
  \end{align*}
  We can exactly compute the first expectation:
  \begin{align*}
    \Ex_z \log \abs{\dotp{v_i, Uz}}^2
    &= \log v_i^\dagger UU^\dagger v_i
      + \Ex_z \log \abs{\dotp{U^\dagger v_i/\norm{U^\dagger  v_i}, z}}^2 \\
    &= \log v_i^\dagger UU^\dagger v_i + \Ex_z \log \abs{z_1}^2 \\
    &= \log v_i^\dagger  P^* v_i - \gamma
  \end{align*}
  Where the first equality follows from normalizing $U^\dagger v_i$, the second
  equality follows from the rotational symmetry of the complex multivariate
  Gaussian since $U^\dagger v_i/\norm{U^\dagger v_i}$ is a unit vector, and the
  third equality follows from Fact \ref{fact:ex-log-CN} for $r=1$. Let
  $\lambda_1, \ldots, \lambda_r$ be the eigenvalues of $U^\dagger U$. Then
  \begin{align*}
    \Ex_z \log z^\dagger U^\dagger Uz = \Ex_z \log \paren{\sum_{i=1}^r \lambda_i \abs{z_i}^2}
    \le \Ex_z \log \paren{\frac{n}{r}\sum_{i=1}^r \abs{z_i}^2} = H_{r-1} - \gamma + \log\pfrac{n}{r},
  \end{align*}
  where the first equality follows from the invariance of the complex
  multivariate Gaussian under unitary transformations (see Proposition
  \ref{prop:cnormal-linear-transform}), and the second equality follows from
  Fact \ref{fact:ex-log-CN}. Next we prove the inequality. Since
  $\Tr(U^\dagger U) = \Tr(P^*) = n$, $\lambda = (\lambda_1, \ldots, \lambda_r)$
  lies on the scaled $r$-simplex. The function
  $\lambda \mapsto \Ex_z\log\paren{\sum_{i=1}^r \lambda_i \abs{z_i}^2}$ is concave
  on the scaled $r$-simplex and is symmetric with respect to all permutations of
  the coordinates of $\lambda$, therefore it is maximized when all
  $\lambda_i = \frac{n}{r}$. Finally we put the above together, along with the
  fact that $\prod_{i=1}^n v_i^\dagger P^* v_i = \nu^*(A)$, to prove the
  theorem.
\end{proof}

\section{Approximating the Permanent} \label{sec:permanent}
We present a new analysis of the relaxation of the permanent of HPSD matrices in
\cite{AnariSimplyExponentialApproximation2017a}. First we show that $\rel(A)$ is
a relaxation of $\per(A)$.
\begin{lemma} \label{lem:per-rel}
  Given any $A \succeq 0$,
  \begin{align*}
    \per(A) \le \rel(A).
  \end{align*}
\end{lemma}
\begin{proof}
  Using the monotonicity of the permanent with respect to the L\"owner order
  (Proposition \ref{prop:per-loewner}), $A \preceq D$ implies that
  $\per(A) \le \per(D)$. Since $D$ is diagonal, $\per(D) = \prod_i D_{ii}$,
  showing that the permanent is always bounded by $\rel(A)$.
\end{proof}
Next we show that $\rel(A)$ is equivalent to the convex relaxation of
\eqref{eq:product-of-linear-forms}.
\begin{lemma}
  \label{lem:rel-sdp-equiv}
  Recall that $\mu^*(A) = \nu^*(A)$ is the optimal value of the convex
  relaxation in Lemma \ref{lem:primal-dual-sdp}. Then
  \begin{align*}
    \rel(A) = \mu^*(A) = \nu^*(A).
  \end{align*}
\end{lemma}
\begin{proof}
  By a scaling argument, the optimum of \eqref{eq:primal-sdp} is achieved when
  $\prod_i \alpha_i = 1$. Taking Schur complements,
  $V \Diag(\alpha) V^\dagger \preceq \lambda I_n$ is equivalent to
  $\lambda \Diag(\alpha)^{-1} \succeq V^\dagger V = A$. Thus by making the
  substitution $D_{ii} = \lambda/\alpha_i$ and noting that
  $\prod_i D_{ii} = \lambda^n$, we show that $\rel(A) = \mu^*(A)$.
\end{proof}
The following lemma shows that given any vector $y \in S_\C(n)$ returned by the
rounding algorithm, we can construct a lower bound on $\per(A)$.
\begin{lemma} \label{lem:per-rank-1} Given HPSD
  $A = V^\dagger V \in \C^{n \times n}$, where $v_i$ are columns of $V$, and
  a vector $y \in S_\C(n)$,
  \begin{align*}
    \frac{n!}{n^n} \prod_{i=1}^n \abs{\dotp{v_i, y}}^2 \le \per(A).
  \end{align*}
\end{lemma}
\begin{proof}
  Since $\norm{y}^2 = n$, $y y^\dagger \preceq nI$ and
  $V^\dagger y y^\dagger V \preceq nV^\dagger V$. Thus
  $n^{-n} \per(V^\dagger y y^\dagger V) \le \per(V^\dagger V) = \per(A)$. Since
  $V^\dagger y y^\dagger V$ is a rank-1 matrix, its permanent is
  $n! \prod_i \abs{\dotp{v_i, y}}^2$ by Proposition \ref{prop:per-rank-1}.
\end{proof}

Now we can state our result about approximating the permanent of a HPSD matrix.
\begin{theorem} \label{thm:approx-factor-prod-linear-complex} Given a HPSD
  matrix $A \in \C^{n \times n}$, $\rel(A)$ is a relaxation of $\per(A)$
  computable in polynomial-time by convex programs \eqref{eq:primal-sdp} or
  \eqref{eq:dual-sdp}. Let $r$ be the rank of $P^*$, the solution to
  \eqref{eq:dual-sdp}. Then $\rel(A)$ is an $\frac{n!}{n^n}e^{-nL_r}$-approximation to
  $\per(A)$:
  \begin{align} \label{eq:approx-factor-prod-linear-complex}
    \frac{n!}{n^n}e^{-nL_r} \rel(A) \le \per(A) \le \rel(A)
  \end{align}
\end{theorem}
Next we state a result that we will use to bound the rank of $P^*$.
\begin{lemma} [Theorem 2.2 in
  \cite{AiLowRankSolutions2008}] \label{lem:rank-reduction} Suppose there is a
  non-zero solution $X$ to the system of equations
  $\{X \succeq 0, \Tr(A_i X) = b_i,\, i = 1, \ldots, d\}$, where $A_i$ is
  Hermitian and $b_i \in \C$. If $d < (r+1)^2$, then one can find in polynomial
  time another solution $X'$ where $\rank(X') = r$.
\end{lemma}
We can now prove Corollary \ref{cor:approx-factor}.
\begin{proof}[Proof of Corollary \ref{cor:approx-factor}]
  Given a solution $P$ to \eqref{eq:dual-sdp}, any HPSD matrix $P'$ that satisfy
  the $n+1$ equalities $\Tr(v_iv_i^\dagger P') = v_i^\dagger P v_i$ and
  $\Tr(P') = n$ will have the same objective value as that of $P$. Applying
  Lemma \ref{lem:rank-reduction}, we can find in polynomial time an optimal
  solution $P^*$ with $\rank(P^*) \le O(\sqrt n)$. We then apply Theorem
  \ref{thm:approx-factor-prod-linear-complex}.
\end{proof}
Finally we prove Theorem \ref{thm:approx-factor-prod-linear-complex}.
\begin{proof}[Proof of Theorem \ref{thm:approx-factor-prod-linear-complex}]
  We use the vector $y$ produced in the rounding procedure in Theorem
  \ref{thm:sdp-rounding} to construct a rank-1 matrix $V^\dagger y y^\dagger
  V$. We then compare the permanent of this matrix to $\per(A)$ and $\rel(A)$:
  \begin{align*}
    \frac{n!}{n^n} e^{-nL_r} \rel(A)
    \stackrel{1}{=} \frac{n!}{n^n}\,  e^{-nL_r} \nu^*(A)
    \stackrel{2}{\le} \frac{n!}{n^n}\, \Ex\bracket{\prod_{i=1}^n \abs{\dotp{v_i, y}}^2}
    \stackrel{3}{\le} \per(A)
    \stackrel{4}{\le} \rel(A)
  \end{align*}
  \begin{enumerate}
  \item Apply Lemma \ref{lem:rel-sdp-equiv}.
  \item Apply Theorem \ref{thm:sdp-rounding}.
  \item Lemma \ref{lem:per-rank-1} shows that for any vector $y \in S_\C(n)$,
    $\per(A) \ge \frac{n!}{n^n} \prod_{i=1}^n \abs{\dotp{v_i, y}}^2$. This is also
    true when taking an expectation of any distribution supported on $S_\C(n)$.
  \item Apply Lemma \ref{lem:per-rel}.
  \end{enumerate}
\end{proof}

\subsection{Low Rank Instances} \label{sec:rank}
There are structured classes of HPSD matrices where we can prove a priori that
the rank of $P^*$ is low and thus a better approximation ratio can be
obtained. For example, it is easy to show that $\rank(P^*) \le \rank(A)$.
Often such instances also have additional symmetry, such as the class of
circulant matrices.
\begin{corollary}
  A square matrix is circulant if each row is cyclically shifted one position to
  the right compared to the previous row. If $A \in \C^{n \times n}$ is HPSD and
  circulant, then there is a solution $P^*$ to \eqref{eq:dual-sdp} where
  $\rank(P^*) = 1$ and we have the bound
  \begin{align*}
    \frac{n!}{n^n}\rel(A) \le \per(A) \le \rel(A).
  \end{align*}
\end{corollary}
\begin{proof}
  Since $A = V^\dagger V$ is circulant it is invariant under the map
  $A_{i,j} \mapsto A_{(i+1 \mod n), (j+1 \mod n)}$. Suppose we have an optimal
  solution $D^*$ to $\rel(A)$ in \eqref{eq:rel-opt-problem}, where $D^*$ is a
  diagonal matrix satisfying $A \preceq D^*$. We then average over all cyclic
  shifts of $D^*$ to show that $D = \lambda I$ is also optimal, which
  corresponds to $\alpha_i=1$ in \eqref{eq:primal-sdp}, with an optimal solution
  $P$ of \eqref{eq:dual-sdp} satisfying the complementary slackness condition of
  $\Tr(PVV^\dagger) = n\lambda$. This shows that $P = vv^\dagger$ is also a
  solution, where $v$ is a suitable multiple of the top eigenvector of
  $VV^\dagger$.
\end{proof}

We also observed experimentally that $\rank(P^*)$ is small for random
$A$. Figure \ref{fig:rank-plots} plots this rank as a function of $n$, for
instances of $A$ drawn from the Gaussian orthogonal ensemble. The results
suggest that $\rank(P^*)$ for these random instances grows slower than
$O(\sqrt n)$.
\begin{figure}[t]
  \label{fig:rank-plots}
  \includegraphics[scale=0.7]{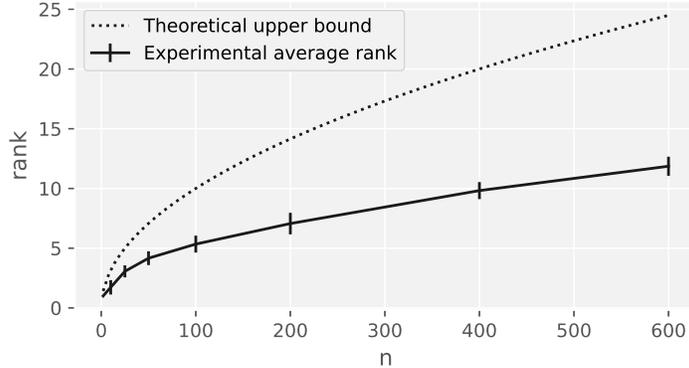}
  \centering
  \caption{Plot of average of $\rank(P^*)$ (along with error bars indicating
    standard deviation) of 50 random instances as a function of $n$.
    $A = V^\dagger V$ is constructed by sampling each entry of
    $V \in \R^{n\times n}$ from a standard Gaussian. The theoretical upper bound
    for $\rank(P^*)$ in Lemma \ref{lem:rank-reduction} is also shown for
    comparison.}
\end{figure}

\section{A Conjecture} \label{sec:conj}
Our analysis of $\rel(A)$ was inspired by the optimization problem
\eqref{eq:product-of-linear-forms}, maximizing a product of linear forms over
the complex sphere. We conjecture that the exact solution to this optimization
problem is a tighter relaxation of the permanent.
\begin{conjecture}\label{conj:psd-per-approx}
  Given $A = V^\dagger V$, where $v_i$ are the columns of $V$, recall that
  $r(A)$ is the maximum of a product of linear forms as defined in
  \eqref{eq:product-of-linear-forms}. Then
  \begin{align} \label{eq:conj} \frac{n!}{n^n} r(A) \le \per(A)
    \le r(A).
  \end{align}
\end{conjecture}
If the matrix $A$ is scaled so that $r(A) = 1$, then \eqref{eq:conj} is exactly
the same bounds given by the van der Waerden's conjecture for doubly stochastic
matrices (proved by \cite{FalikmanProofvanWaerden1981},
\cite{EgorychevsolutionvanWaerden1981}
\cite{GurvitsVanWaerdenSchrijverValiant2008a}). The lower bound follows from
Lemma \ref{lem:per-rank-1}, but the upper bound cannot be proven by naively
applying Proposition \ref{prop:per-int-rep} and bounding the integral over the
complex sphere by its maximum. However, we can show that the upper bound is
implied by another conjecture on permanents:
\begin{conjecture}[Pate's conjecture \cite{Pateinequalityinvolvingpermanents1984}]
  Given any $n \times n$ HPSD matrix $A$, let $A \otimes J_k$ be the Kronecker
  product of $A$ with the $k \times k$ all-ones matrix. Then
  \begin{align} \label{eq:pate-conj}
    \per(A \otimes J_k) \ge  \per(A)^k (k!)^n.
  \end{align}
\end{conjecture}
This conjecture has been proved in the case where $n = 2$, see
\cite{Zhangupdatefewpermanent2016} for a survey of subsequent progress on this
conjecture. Using the integral representation of the permanent (Proposition
\ref{prop:per-int-rep}), we can write \eqref{eq:pate-conj} as:
\begin{align*}
  \Ex_{x\sim \cnormal(0, I_n)}\bracket{\prod_{i=1}^n \abs{\dotp{v_i,x}}^{2k}}^{1/k} \ge \per(A)
  \Ex_{x\sim \cnormal(0, I_n)}\bracket{\prod_{i=1}^n \abs{x_i}^{2k}}^{1/k}
\end{align*}
Since both expectations are taken over homogeneous polynomials of degree $d$, we
can apply Fact \ref{fact:gaussian-sphere-correction}, take
$k \rightarrow \infty$ and get:
\begin{align*}
  \max_{\norm{x}^2 = n} \prod_{i=1}^n \abs{\dotp{v_i, x}}^2
  \ge \per(A) \max_{\norm{x}^2 = n} \prod_{i=1}^n \abs{x_i}^2 = \per(A).
\end{align*}

\section{Discussion and Conclusion}
There are a few interesting directions that stem from this work. For random $A$
(i.e. drawn from the Gaussian orthogonal ensemble), numerical experiments in
Section \ref{sec:rank} suggest that $\rank(P^*)$ is very small compared to
$\sqrt{n}$. It would be interesting to provide concrete bounds on the rank of
random instances. One might also ask if we can construct sequences of matrices
$A_k$ of increasing size but with fixed rank $r$, where
$(\rel(A_k)/\per(A_k))^{1/n} \rightarrow e^{1+L_r}$. This is related to the
question called the \emph{linear polarization constant of Hilbert spaces}, see
\cite{PappasLinearpolarizationconstants2004} for such a construction and its
analysis.

The main result of this paper uses the connection between the permanent and the
optimization of a product of linear forms over the sphere
\eqref{eq:product-of-linear-forms}. Although it is natural to conjecture the
hardness of computing $r(A)$, we do not know of any formal results establishing
this. We also proposed Conjecture \ref{conj:psd-per-approx} which would explain
why this optimization problem is intimately related to the permanent. Better
understanding of this problem may lead to further insights about the permanent
of HPSD matrices.

\bibliographystyle{amsalpha}
\bibliography{../../MyLibraryTEX}

\appendix
\section{Asymptotics of the Approximation Factor}
\begin{proposition}\label{prop:approx-asymptotics}
  For all positive integers $r$,
  \begin{align} \label{eq:gamma-asymp}
    \frac{1}{2r} < \gamma - L_r < \frac{r+2}{2r(r+1)}.
  \end{align}
\end{proposition}
\begin{proof}
  It is easy to see that \eqref{eq:gamma-asymp} follows from
  \begin{align*}
    \frac{1}{2(r+1)} < H_n - \log(r) - \gamma < \frac{1}{2r}.
  \end{align*}
  From Figure \ref{fig:asymp}, we can see that
  $H_n - \log(r) - \gamma = \sum_{k=r}^\infty \Delta_k$.  The upper bound is
  given by computing the sum of the areas of the larger triangles:
  \begin{align*}
    \sum_{k=r}^\infty \Delta_k < \sum_{k=r}^\infty \frac{1}{2}\paren{\frac{1}{k} - \frac{1}{k+1}}
    = \sum_{k=r}^\infty \frac{1}{2k(k+1)} = \frac{1}{2r}
  \end{align*}
  The lower bound is given by computing the sum of the areas of the smaller triangles:
  \begin{align*}
    \sum_{k=r}^\infty \Delta_k > \sum_{k=r}^\infty \frac{1}{2}\frac{1}{(k+1)^2}
    > \sum_{k=r}^\infty \frac{1}{2(k+1)(k+2)} = \frac{1}{2(r+1)}
  \end{align*}
  \begin{figure}[h]
    \centering
    \def\svgwidth{\columnwidth/3*2}
\begingroup%
  \makeatletter%
  \providecommand\color[2][]{%
    \errmessage{(Inkscape) Color is used for the text in Inkscape, but the package 'color.sty' is not loaded}%
    \renewcommand\color[2][]{}%
  }%
  \providecommand\transparent[1]{%
    \errmessage{(Inkscape) Transparency is used (non-zero) for the text in Inkscape, but the package 'transparent.sty' is not loaded}%
    \renewcommand\transparent[1]{}%
  }%
  \providecommand\rotatebox[2]{#2}%
  \newcommand*\fsize{\dimexpr\f@size pt\relax}%
  \newcommand*\lineheight[1]{\fontsize{\fsize}{#1\fsize}\selectfont}%
  \ifx\svgwidth\undefined%
    \setlength{\unitlength}{267.95343215bp}%
    \ifx\svgscale\undefined%
      \relax%
    \else%
      \setlength{\unitlength}{\unitlength * \real{\svgscale}}%
    \fi%
  \else%
    \setlength{\unitlength}{\svgwidth}%
  \fi%
  \global\let\svgwidth\undefined%
  \global\let\svgscale\undefined%
  \makeatother%
  \begin{picture}(1,0.45717719)%
    \lineheight{1}%
    \setlength\tabcolsep{0pt}%
    \put(0,0){\includegraphics[width=\unitlength,page=1]{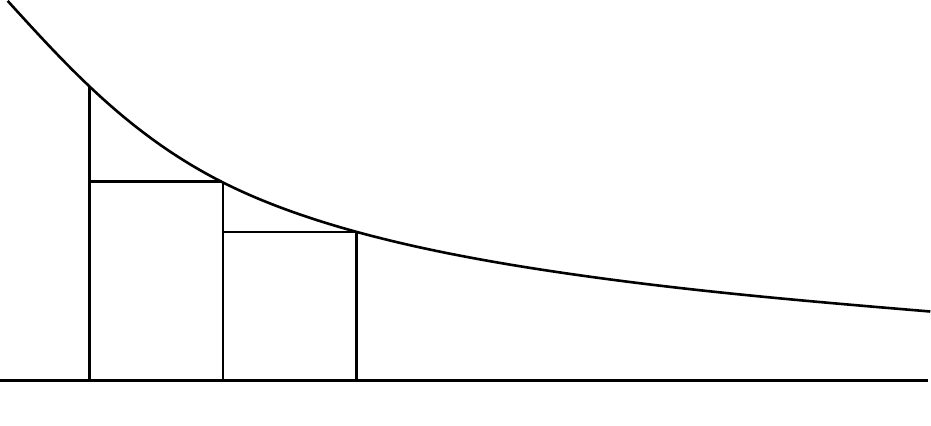}}%
    \put(0.45208007,0.24631129){\color[rgb]{0,0,0}\makebox(0,0)[lt]{\smash{\begin{tabular}[t]{l}...\end{tabular}}}}%
    \put(0.08921057,0.00339855){\color[rgb]{0,0,0}\makebox(0,0)[lt]{\smash{\begin{tabular}[t]{l}$r$\end{tabular}}}}%
    \put(0.20801293,0.00024601){\color[rgb]{0,0,0}\makebox(0,0)[lt]{\smash{\begin{tabular}[t]{l}$r+1$\end{tabular}}}}%
    \put(0.35274606,0){\color[rgb]{0,0,0}\makebox(0,0)[lt]{\smash{\begin{tabular}[t]{l}$r+2$\end{tabular}}}}%
    \put(0.15118825,0.33627894){\color[rgb]{0,0,0}\makebox(0,0)[lt]{\smash{\begin{tabular}[t]{l}$\Delta_r$\end{tabular}}}}%
    \put(0.29554656,0.26243073){\color[rgb]{0,0,0}\makebox(0,0)[lt]{\smash{\begin{tabular}[t]{l}$\Delta_{r+1}$\end{tabular}}}}%
    \put(0,0){\includegraphics[width=\unitlength,page=2]{gamma-asymp-src.pdf}}%
    \put(0.6819793,0.09351196){\color[rgb]{0,0,0}\makebox(0,0)[lt]{\smash{\begin{tabular}[t]{l}$y=\frac{1}{x}$\\\end{tabular}}}}%
  \end{picture}%
\endgroup%

    \caption{\label{fig:asymp} Illustration of the asymptotics of
      $L_r$. $\Delta_r$ is the area between the curve $y=\frac{1}{x}$ and the
      rectangle of height $\frac{1}{r+1}$. The lower dotted lines are tangent to
      the curve at $r+1$ and $r+2$ respectively. }
  \end{figure}
\end{proof}
\end{document}